\documentclass[oneside,english]{amsart}
\usepackage[T1]{fontenc}
\usepackage[latin9]{inputenc}
\usepackage{amsthm}
\usepackage{amstext}
\usepackage{amssymb}
\usepackage{esint}

\makeatletter
\numberwithin{equation}{section}
\numberwithin{figure}{section}
\theoremstyle{plain}
\newtheorem{thm}{\protect\theoremname}[section]
  
  \theoremstyle{plain}
  \theoremstyle{definition}
  \newtheorem{defn}[thm]{\protect\definitionname}
  \theoremstyle{plain}
  \newtheorem{lem}[thm]{\protect\lemmaname}
  
  \theoremstyle{plain}
	\newtheorem{rem}[thm]{\protect\remarkname}
  \theoremstyle{plain}
	\theoremstyle{plain}
	\newtheorem{exa}[thm]{\protect\examplename}
\makeatother

\usepackage{babel}
  \providecommand{\definitionname}{Definition}
  \providecommand{\lemmaname}{Lemma}
  \providecommand{\theoremname}{Theorem}
  \providecommand{\corollaryname}{Corollary}
  \providecommand{\remarkname}{Remark}
  \providecommand{\propositionname}{Proposition}
  \providecommand{\examplename}{Example}
	
\begin{document}

\title[Spectral estimates of the $p$-Laplace Neumann operator]{Spectral estimates of the $p$-Laplace Neumann operator in conformal regular domains}

\author{V.~Gol'dshtein and A.~Ukhlov}
\begin{abstract}
In this paper we study spectral estimates of the $p$-Laplace Neumann operator in conformal regular domains $\Omega\subset\mathbb R^2$. 
This study is based on (weighted) Poincar\'e-Sobolev inequalities. The main technical tool is the composition operators theory in relation with the Brennan's conjecture. We prove that if the Brennan's conjecture holds then for any $p\in (4/3,2)$ and $r\in (1,p/(2-p))$ the weighted $(r,p)$-Poincare-Sobolev inequality holds with the constant depending on the conformal geometry of $\Omega$. As a consequence we obtain classical Poincare-Sobolev inequalities and spectral estimates for the first nontrivial eigenvalue of the $p$-Laplace Neumann operator for conformal regular domains.

\end{abstract}
\maketitle

\section{Introduction and methodology}

Let  $\Omega\subset\mathbb R^2$ be a simply connected planar domain with a smooth boundary $\partial \Omega$.  We consider the Neumann eigenvalue problem
for the $p$-Laplace operator ($1<p<2$):
\begin{equation}
\label{pLaplace}
\begin{cases}
-\operatorname{div}\left(|\nabla u|^{p-2}\nabla u\right)=\mu_p|u|^{p-2}u& \text{in}\,\,\,\,\,\Omega\\
\frac{\partial u}{\partial n}=0& \text{on}\,\,\,\partial\Omega.
\end{cases}
\end{equation}

The weak statement of this spectral problem is as follows: a function
$u$ solves the previous problem iff $u\in W^{1,p}(\Omega)$ and 
\[
\iint\limits _{\Omega} \left(|\nabla u(x,y)|^{p-2}\nabla u(x,y)\right)\cdot\nabla v(x,y)~dxdy=\mu_p \iint\limits _{\Omega}|u|^{p-2}u(x,y)v(x,y)~dxdy
\]
for all $v\in W^{1,p}(\Omega)$. 

The first nontrivial Neumann eigenvalue $\mu_p$ can be characterized as
$$
\mu_p(\Omega)=\min\left\{\frac{\iint\limits_{\Omega}|\nabla u(x,y)|^p~dxdy}{\iint\limits_{\Omega}|u(x,y)|^p~dxdy}: u\in W^{1,p}(\Omega)\setminus\{0\},\,
\iint\limits_{\Omega}|u|^{p-2}u~dxdy=0\right\}.
$$

Moreover, $\mu_p(\Omega)^{-\frac{1}{p}}$ is the best constant $B_{p,p}(\Omega)$ (see, for example, \cite{BCT15}) in the following Poincar\'e-Sobolev inequality
$$
\inf\limits_{c\in\mathbb R}\|f-c \mid L^p(\Omega)\|\leq B_{p,p}(\Omega)\|\nabla f \mid L^p(\Omega)\|,\,\,\, f\in W^{1,p}(\Omega).
$$

We prove, that $\mu_p(\Omega)$ depends on the conformal geometry of $\Omega$ and can be estimated in terms of Sobolev norms of a conformal mapping of the unit disc $\mathbb D$ onto $\Omega$  (Theorem A).

The main technical tool is existence of universal weighted Poincar\'e-Sobolev inequalities 
\begin{multline}
\inf_{c\in\mathbb R}\biggl(\iint\limits _{\Omega}|f(x,y)-c|^{r}h(x,y)\, dxdy \biggr)^{\frac{1}{r}}\\
\leq B_{r,p}(\Omega,h)\biggl(\iint\limits _{\Omega}
|\nabla f(x,y)|^{p}~dxdy\biggr)^{\frac{1}{p}},\,\,f\in W^{1,p}(\Omega),
\label{eq:WPI}
\end{multline} 
in any simply connected domain $\Omega\ne\mathbb R^2$ for conformal weights $h(x,y):=J_{\varphi}(x,y)=|\varphi^{\prime} (x,y)|^2$ induced by conformal homeomorphisms $\varphi:\Omega\to\mathbb D$. 

\vspace{0.5cm}

{\sc Main results of this article can be divided onto two parts. The first part is the technical one and concerns to weighted Poincar\'e-Sobolev inequalities in arbitrary simply connected planar domains with non\-emp\-ty boundaries (Theorem C and its consequences). Results of the first part will be used for (non weighted) Poincar\'e-Sobolev inequalities in so-called conformal regular domains (Theorem B) that leads to lower estimates for the first nontrivial eigenvalue $\mu_p$ (Theorem A). To the best of our knowledge lower estimates were known before for convex domains only. The class of conformal regular domains is much larger. It includes, for example, bounded domains with Lipschitz boundaries and quasidiscs, i.e images of discs under quasiconformal homeomorphisms of whole plane.}

\vspace{0.5cm}

{\bf Brennan's conjecture} \cite{Br} is that for a conformal mapping $\varphi : \Omega\to\mathbb D$
\begin{equation}
\int\limits _{\Omega}|\varphi'(x,y)|^{\beta}~dxdy<+\infty,\quad\text{for all}\quad\frac{4}{3}<\beta<4.\label{eq:BR1}
\end{equation}
For the inverse conformal mapping  $\psi=\varphi^{-1} : \mathbb D\to \Omega$ Brennan's conjecture \cite{Br} states 
\begin{equation}
\iint\limits _{\mathbb D}|\psi'(u,v)|^{\alpha}~dudv <+\infty,\quad\text{for all}\quad -2<\alpha<\frac{2}{3}.\label{eq:BR}
\end{equation}

\vspace{0.5cm}

A connection between Brennan's Conjecture and composition operators on Sobolev spaces was established in \cite{GU3}:

\vskip 0.5cm
{\bf Equivalence Theorem.} 
{\it Brennan's Conjecture (\ref{eq:BR1}) holds for a number $s\in ({4}/{3};4)$ if and only if a conformal 
mapping $\varphi : \Omega\to\mathbb D$ induces a bounded composition operator
$$
\varphi^{\ast}: L^{1,p}(\mathbb D)\to L^{1,q(p,\beta)}(\Omega)
$$
for any $p\in (2;+\infty)$ and $q(p,\beta)=p\beta/(p+\beta-2)$.}

\vspace{0.5cm}

The inverse Brennan's Conjecture states that for any conformal mapping $\psi:\mathbb D\to\Omega$, the derivative $\psi'$ belongs to the Lebesgue space $L^{\alpha}(\mathbb D)$, for $-2<\alpha<2/3$.  The integrability of the derivative in the power greater then $2/3$ requires some restrictions on the geometry of $\Omega$. If $\Omega\subset\mathbb R^2$ is a simply connected planar domain of finite area, then
$$
\iint\limits_{\mathbb D} |\psi'(u,v)|^2~dudv= \iint\limits_{\mathbb D} J_{\psi}(u,v)~dudv=|\Omega|<\infty.
$$
Integrability of the derivative in the power $\alpha>2$ is impossible without additional assumptions on the geometry of $\Omega$. For example, for any  $\alpha>2$ the domain $\Omega$  necessarily has a finite geodesic diameter \cite{GU4}. 

\vspace{0.5cm}

{\it Let  $\Omega\subset\mathbb{R}^2$ be a simply connected planar domain. Then 
$\Omega$ is called a conformal $\alpha$-regular domain if there exists a conformal mapping $\varphi:\Omega\to\mathbb{D}$ such that
\begin{equation}
\label{alphareg}
\iint\limits_{\mathbb{D}}\left|(\varphi^{-1})'(u,v)\right|^{\alpha}~dudv<\infty\,\,\,\text{for some}\,\,\, \alpha>2.
\end{equation}
If $\Omega$ is a conformal $\alpha$-regular domain for some $\alpha>2$  we call $\Omega$ a conformal regular domain.}

The property of $\alpha$-regularity does not depends on choice of a conformal mapping $\varphi$ and depends on the hyperbolic geometry of $\Omega$ only. For connection between conformal mapping and hyperbolic geometry see, for example, \cite{BM}.

Note that a boundary $\partial \Omega$ of a conformal regular domain can have any Hausdorff dimension between one and two, but can not be equal two \cite{HKN}.

The next theorem gives lower estimates of the first nontrivial $p$-Laplace Neumann eigenvalue: 

\vspace{0.5cm}

{\bf Theorem A.} {\it Let $\varphi:\Omega\to\mathbb{D}$ be a conformal homeomorphism from a conformal $\alpha$-regular domain
 $\Omega$ to the unit disc $\mathbb{D}$ and  Brennan's Conjecture holds. Then for every 
$p \in \left( \max\{4/3, (\alpha+2)/\alpha\},2 \right)$  the following inequality is correct
\begin{multline}
\frac{1}{\mu_p(\Omega)}\leq\\
 \inf\limits_{q\in[1,2p/(4-p))} \left\{\|{(\varphi^{-1})^{\prime}}| L^{\alpha}(\mathbb D)\|^{2}\left(\iint\limits_{\mathbb D}|\left(\varphi^{-1}\right)^{\prime}|^{\frac{(p-2)q}{p-q}}~dudv\right)^{\frac{p-q}{q}} \cdot B^{p}_{\frac{\alpha p}{\alpha-2},q}(\mathbb D)\right\}.
\nonumber
\end{multline}
Here $B_{r,q}(\mathbb D)$ is the exact constant in the corresponding $(r,q)$-Poincare-Sobolev inequality in the unit disc $\mathbb D$ for $r=\alpha p/(\alpha-2)$.
}  

In the limit case $\alpha=\infty$ and $p=q$ we have

\vspace{0.5cm}

{\bf Corollary A.} {\it Suppose that $\Omega$ is smoothly bounded Jordan domain with a boundary $\partial\Omega$ of a class $C^1$ with a Dini continuous normal. Let $\varphi:\Omega\to\mathbb{D}$ be a conformal homeomorphism from $\Omega$ onto the unit disc $\mathbb{D}$. Then for every 
$p \in \left( 1,2 \right)$  the following inequality is correct
$$
\frac{1}{\mu_p(\Omega)}\leq
\|{(\varphi^{-1})^{\prime}}| L^{\infty}(\mathbb D)\|^{p}\frac{1}{\mu_p(\mathbb D)}.
$$
}  

\begin{rem}
The constant $B_{r,q}(\mathbb D)$ satisfies \cite{GT, GU5}:
$$
B_{r,q}(\mathbb D)\leq \frac{2}{\pi^{\delta}}\left(\frac{1-\delta}{{1}/{2}-\delta}\right)^{1-\delta},\,\,\delta={1}/{q}-{1}/{r}.
$$
\end{rem}

\begin{rem}
The Brennan's conjecture was proved for $\alpha\in [\alpha_0,2/3)$ when $\alpha_0=-1.752$ \cite{Shim}. 
\end{rem}

\begin{rem}
The numeric estimates for the $\mu_p(\Omega)$ were known before only for convex domains. For example, in \cite{ENT} was proved that
$$
\mu_p(\Omega)\geq \left(\frac{\pi_p}{d(\Omega)}\right)^p
$$
where
$$
\pi_p=2\int\limits_0^{(p-1)^{\frac{1}{p}}}\frac{dt}{(1-t^p/(p-1))^{\frac{1}{p}}}
=2\pi\frac{(p-1)^{\frac{1}{p}}}{p(\sin(\pi/p))}.
$$
\end{rem}

\begin{rem}
Theorem A has a direct connection with the spectral stability problem for the $p$-Laplace operator. 
See, the recent papers, \cite{BL, BLC, L}, where one can found the history of the problem, main results in this area and appropriate references.
\end{rem}

Theorem A is a corollary  (after simple calculations)  of the following version of the 
Poincar\'e-Sobolev inequality

\vspace{0.5cm}

{\bf Theorem B.}
{\it Suppose that $\Omega\subset\mathbb{R}^2$ be a conformal $\alpha$-regular domain and Brennan's Conjecture holds.
Then for every $p \in \left( \max\{4/3, \alpha/(\alpha-1)\},2 \right)$, every $s \in (1,\frac{\alpha-2}{\alpha}\frac{p}{2-p})$ and every function 
$f\in W^{1,p}(\Omega)$, the inequality 
\begin{equation}
\inf_{c\in\mathbb R}\biggl(\iint\limits _{\Omega}|f(x,y)-c|^{s}\, dxdy \biggr)^{\frac{1}{s}}\leq B_{s,p}(\Omega)\biggl(\iint\limits _{\Omega}
|\nabla f(x,y)|^{p}~dxdy\biggr)^{\frac{1}{p}}\label{eq:NonWPI}
\end{equation}
holds with the constant 
\begin{multline}
B_{s,p}(\Omega)\leq \|{(\varphi^{-1})^{\prime}}| L^{\alpha}(\mathbb D)\|^{\frac{2}{s}}B_{r,p}(\Omega,h)\\
 \leq \inf\limits_{q\in[1,2p/(4-p))}\left\{B_{\frac{\alpha s}{\alpha-2},q}(\mathbb D)\cdot\|{(\varphi^{-1})^{\prime}}| L^{\alpha}(\mathbb D)\|^{\frac{2}{s}}K_{p,q}(\mathbb D) \right\}.
\nonumber
\end{multline}
}

Here $B_{r,p}(\Omega,h)$, $r=\alpha s/(\alpha-2)$,  is the best constant of the following weighted  Poincar\'e-Sobolev inequality:

\vspace{0.5cm}

{\bf Theorem C.}
{\it Suppose  $\Omega\subset\mathbb{R}^2$ is a simply connected domain with non empty boundary, Brennan's Conjecture holds and $h(z)=J(z,\varphi)$ is the conformal weight defined by a conformal homeomorphism $\varphi : \Omega\to\mathbb D$. Then for every $p \in \left( 4/3,2 \right)$ and every function 
$f\in W^{1,p}(\Omega)$, the inequality 
\begin{equation}
\inf_{c\in\mathbb R}\biggl(\iint\limits _{\Omega}|f(x,y)-c|^{r}h(x,y)\, dxdy \biggr)^{\frac{1}{r}}\leq B_{r,p}(\Omega,h)\biggl(\iint\limits _{\Omega}
|\nabla f(x,y)|^{p}~dxdy\biggr)^{\frac{1}{p}}
\label{eq:WPI1}
\end{equation} 
holds for any $r \in [1,p/(2-p))$ with the constant 
$$
B_{r,p}(\Omega,h)\leq \inf\limits_{q\in[1,2p/(4-p))}\left\{K_{p,q}(\mathbb D)\cdot B_{r,q}(\mathbb D)\right\}
$$.}
%\end{thm}

Here $B_{r,q}(\mathbb D)$ is the best constant in the (non-weighted) $(r.q)$-Poincar\'e-Sobolev inequality in the unit disc 
$\mathbb D\subset\mathbb R^2$ and $K_{p,q}(\Omega)$ is the norm of composition operator 
$$
\left(\varphi^{-1}\right)^{\ast}: L^{1,p}(\Omega)\to L^{1,q}(\mathbb D)
$$
generated by the inverse conformal mapping $\varphi^{-1}: \mathbb D\to\Omega$:
$$
K_{p,q}(\Omega)\leq \left(\iint\limits_{\mathbb D}|\left(\varphi^{-1}\right)'|^{\frac{(p-2)q}{p-q}}~dudv\right)^{\frac{p-q}{pq}}.
$$

\begin{rem}
 Theorem C is correct (without referring of Brennan's conjecture) for 
\begin{equation}
1\leq r\leq \frac{2p}{2-p}\cdot\frac{\left|\alpha_{0}\right|}{2+\left|\alpha_{0}\right|}<\frac{p}{2-p}\label{eq:PQR}
\end{equation}
and $p\in ((|\alpha_0|+2)/(|\alpha_0|+1),2)$, where $\alpha_0=-1.752$ represents the best result for which Brennan's conjecture was proved.
\end{rem}

\begin{rem} Let $\Omega\subset\mathbb R^2$ be a simply connected smooth domain. Then $\varphi^{-1}\in L^{\alpha}(\mathbb D)$ for all $\alpha\in\mathbb R$ and we have the weighted Poincar\'e-Sobolev inequality (\ref{eq:WPI}) for all $p\in[1,2)$ and all $r\in [1,2p/(2-p)]$.
\end{rem}

In the case, when we have embedding of weighted Lebesgue spaces in non-weighted one, the weighted Poincar\'e-Sobolev inequality 
(\ref{eq:WPI1}) implies the standard Poin\-ca\-r\'e-Sobolev inequality (\ref{eq:NonWPI}). 

Let us give some historical remarks about the notion of conformal regular domains.
This notion was introduced in \cite{BGU1} and was applied to the stability problem for eigenvalues of the Dirichlet-Laplace operator. 
In \cite{GU5}  the lower estimates for the first non-trivial eigenvalues of the the Neumann-Laplace operator in conformal regular domains  were obtained.
In \cite{GU4} we proved but did not formulated the following important fact about conformal regular domains and the Poinca\-r\'e-Sobolev inequality:

\begin{thm}
Let a simply connected domain $\Omega\subset\mathbb R^2$ of finite area does not support the (s,2)-Poincar\'e-Sobolev inequality
$$
\inf\limits_{c\in\mathbb R}\biggl(\iint\limits_{\Omega}|f(x,y)-c|^{s}~dxdy\biggr)^{\frac{1}{s}}\leq B_{s,2}(\Omega)\biggl(\iint\limits_{\Omega}|\nabla f(x,y)|^2~dxdy\biggr)^{\frac{1}{2}}
$$
for some $s\geq 2$. Then $\Omega$ is not a conformal regular domain.
\end{thm}

In the present work we suggest for the conformal regular domains a new method based on the composition operators theory. The suggested method of investigation is based on the composition
operators theory \cite{U1,VU1} and its applications to the Sobolev
type embedding theorems \cite{GGu,GU}. 

The following diagram illustrate this idea:

\[\begin{array}{rcl}
W^{1,p}(\Omega) & \stackrel{\varphi^*}{\longrightarrow} & W^{1,q}(\mathbb{D}) \\[2mm]
\multicolumn{1}{c}{\downarrow} & & \multicolumn{1}{c}{\downarrow} \\[1mm]
L^s(\Omega) & \stackrel{(\varphi^{-1})^*}{\longleftarrow} & L^r(\mathbb{D})
\end{array}\]

Here the operator $\varphi^{\ast}$ defined by the composition rule $\varphi^{\ast}(f)=f\circ\varphi$ is a bounded composition operator on Sobolev spaces induced by a homeomorphism $\varphi$ of $\Omega$ and $\mathbb D$ and the operator $(\varphi^{-1})^{\ast}$ defined by the composition rule $(\varphi^{-1})^{\ast}(f)=f\circ\varphi^{-1}$ is a bounded composition operator on Lebesgue spaces.
This method allows to transfer Poinca\-r\'e-Sobolev inequalities from regular domains (for example, from the unit disc $\mathbb D$) to $\Omega$.

In the recent works we studied
composition operators on Sobolev spaces 
in connection with the conformal mappings theory \cite{GU1}. This
connection leads to weighted Sobolev embeddings \cite{GU2,GU3} with
the universal conformal weights. Another application of conformal
composition operators was given in \cite{BGU1} where the spectral
stability problem for conformal regular domains was considered.

\section{Composition Operators}

Because all composition operators that will be used in this paper are induced by conformal homeomorphisms we formulate results about composition operators for diffeomorphisms only. 

\subsection{Composition Operators on Lebesgue Spaces}

For any domain $\Omega \subset \mathbb{R}^{2}$ and any $1\leq p<\infty$
we consider the Lebesgue space 
\[
L^{p}(\Omega):=\left\{ f:\Omega \to R:\|f\mid{L^{p}(\Omega)}\|:=\left(\iint\limits _{\Omega}|f(x,y)|^{p}~dxdy\right)^{1/p}<\infty\right\} .
\]
 
The following theorem  about composition operators on Lebesgue spaces is well known (see, for example \cite{VU1}): 

\begin{thm}
\label{thm:LpLq} Let $\varphi:\Omega\to\ \Omega'$  be a diffeomorphism between two planar domains $\Omega$ and $\Omega'$. Then the
composition operator 
\[
\varphi^{\ast}:L^{r}(\Omega')\to L^{s}(\Omega),\,\,\,1\leq s\leq r<\infty,
\]
 is bounded, if and only if 
and 
\begin{gather}
\biggl(\iint\limits _{\Omega'}\left(J_{\varphi^{-1}}(u,v)\right)^{\frac{r}{r-s}}~dudv\biggl)^{\frac{r-s}{rs}}=K<\infty, \,\,\,1\leq s<r<\infty,\nonumber\\
\sup\limits_{(u,v)\in\Omega'}\left(J_{\varphi^{-1}}(u,v)\right)^{\frac{1}{s}}=K<\infty, \,\,\,1\leq s=r<\infty.
\nonumber
\end{gather}
 The norm of the composition operator $\|\varphi^{\ast}\|=K$.
\end{thm}

\subsection{Composition Operators on Sobolev Spaces}

We define the Sobolev space $W^{1,p}(\Omega)$, $1\leq p<\infty$
as a Banach space of locally integrable weakly differentiable functions
$f:\Omega\to\mathbb{R}$ equipped with the following norm: 
\[
\|f\mid W^{1,p}(\Omega)\|=\biggr(\iint\limits _{\Omega}|f(x,y)|^{p}\, dxdy\biggr)^{\frac{1}{p}}+\biggr(\iint\limits _{\Omega}|\nabla f(x,y)|^{p}\, dxdy\biggr)^{\frac{1}{p}}.
\]

We define also the homogeneous seminormed Sobolev space $L^{1,p}(\Omega)$
of locally integrable weakly differentiable functions $f:\Omega\to\mathbb{R}$ equipped
with the following seminorm: 
\[
\|f\mid L^{1,p}(\Omega)\|=\biggr(\iint\limits _{\Omega}|\nabla f(x,y)|^{p}\, dxdy\biggr)^{\frac{1}{p}}.
\]

Recall that the embedding operator $i:L^{1,p}(\Omega)\to L_{\operatorname{loc}}^{1}(\Omega)$
is bounded.

Let $\Omega$ and $\Omega'$ be domains in $\mathbb{R}^2$. We say that
a diffeomorphism $\varphi:\Omega\to\Omega'$ induces a bounded composition
operator 
\[
\varphi^{\ast}:L^{1,p}(\Omega')\to L^{1,q}(\Omega),\,\,\,1\leq q\leq p\leq\infty,
\]
by the composition rule $\varphi^{\ast}(f)=f\circ\varphi$ if $\varphi^{\ast}(f)\in L^{1,q}(\Omega)$ and there exists a constant $K<\infty$ such that 
\[
\|\varphi^{\ast}(f)\mid L^{1,q}(\Omega)\|\leq K\|f\mid L^{1,p}(\Omega')\|.
\]

The main result of \cite{U1} gives the analytic description of composition
operators on Sobolev spaces (see, also \cite{VU1}) and asserts (in the case of diffeomorphisms) that

\begin{thm}
\label{CompTh} \cite{U1} A diffeomorphism $\varphi:\Omega\to\Omega'$
between two domains $\Omega$ and $\Omega'$ induces a bounded composition
operator 
\[
\varphi^{\ast}:L^{1,p}(\Omega')\to L^{1,q}(\Omega),\,\,\,1\leq q<p<\infty,
\]
 if and only if 
\[
K_{p,q}(\Omega)=\biggl(\iint\limits _{\Omega}\biggl(\frac{|\varphi'(x,y)|^{p}}{|J_{\varphi}(x,y)|}\biggr)^{\frac{q}{p-q}}~dxdy\biggr)^{\frac{p-q}{pq}}<\infty.
\]
\end{thm}

\begin{defn}
We call a bounded domain $\Omega\subset\mathbb{R}^2$ as
a $(r,q)$-embedding domain, $1\leq q, r\leq\infty$, if the embedding operator
$$
i_{\Omega}: W^{1,q}(\Omega)\hookrightarrow L^r(\Omega)
$$
is bounded. The unit disc $\mathbb D\subset\mathbb R^2$ is an example of the $(r,2)$-embedding domain for all $r\geq 1$. 
\end{defn}

The following theorem gives a characterization of composition operators in the normed Sobolev spaces\cite{GU5}. For readers convenience we reproduce here the proof of the theorem.

\begin{thm}
\label{thm:boundWW} Let $\Omega$ be an $(r,q)$-embedding domain for some $1\leq q\leq r\leq\infty$ and $|\Omega'|<\infty$.
Suppose that a diffeomorphism $\varphi:\Omega\to\Omega'$ induces
a bounded composition operator 
\[
\varphi^{\ast}:L^{1,p}(\Omega')\to L^{1,q}(\Omega),\,\,\,1\leq q\leq p<\infty,
\]
 and the inverse diffeomorphism $\varphi^{-1}:\Omega'\to\Omega$ induces
a bounded composition operator 
\[
(\varphi^{-1})^{\ast}:L^{r}(\Omega)\to L^{s}(\Omega')
\]
for some $p\leq s\leq r$.

 Then $\varphi:\Omega\to\Omega'$ induces
a bounded composition operator 
\[
\varphi^{\ast}:W^{1,p}(\Omega')\to W^{1,q}(\Omega),\,\,\,1\leq q\leq p<\infty.
\]
 \end{thm}

\begin{proof}
Because the composition operator $(\varphi^{-1})^{\ast}:L^{r}(\Omega)\to L^{s}(\Omega')$
is bounded, then the following inequality 
\[
\|(\varphi^{-1})^{\ast}g\mid L^{s}(\Omega')\|\leq A_{r,s}(\Omega)\|g\mid L^{r}(\Omega)\|
\]
 is correct. Here $A_{r,s}(\Omega)$ is a positive constant.

If a domain $\Omega$ is an embedding domain and the composition
operators 
\[
(\varphi^{-1})^{\ast}:L^{r}(\Omega)\to L^{s}(\Omega'),\,\,\,\varphi^{\ast}:L^{1,p}(\Omega')\to L^{1,q}(\Omega)
\]
 are bounded, then for a function $f=g\circ\varphi^{-1}$ the following
inequalities 
\begin{multline*}
\inf\limits _{c\in\mathbb{R}}\|f-c\mid L^{s}(\Omega')\|\leq A_{r,s}(\Omega)\inf\limits _{c\in\mathbb{R}}\|g-c\mid L^{r}(\Omega)\|\\
\leq A_{r,s}(\Omega)M\|g\mid L^{1,q}(\Omega)\|\leq A_{r,s}(\Omega)K_{p,q}(\Omega)M\|f\mid L^{1,p}(\Omega')\|
\end{multline*}
 hold. Here $M$ and $K_{p,q}(\Omega)$ are positive constants.

The H\"older inequality implies the following estimate 
\begin{multline*}
|c|=|\Omega'|^{-\frac{1}{p}}\|c\mid L^{p}(\Omega')\|\leq|\Omega'|^{-\frac{1}{p}}\bigl(\|f\mid L^{p}(\Omega')\|+\|f-c\mid L^{p}(\Omega')\|\bigr)\\
\leq|\Omega'|^{-\frac{1}{p}}\|f\mid L^{p}(\Omega')\|+|\Omega'|^{-\frac{1}{s}}\|f-c\mid L^{s}(\Omega')\|.
\end{multline*}

Because $q\leq r$ we have 
\begin{multline*}
\|g\mid L^{q}(\Omega)\|\leq\|c\mid L^{q}(\Omega)\|+\|g-c\mid L^{q}(\Omega)\|\leq|c||\Omega|^{\frac{1}{q}}+|\Omega|^{\frac{r-q}{r}}\|g-c\mid L^{r}(\Omega)\|\\
\leq\biggl(|\Omega'|^{-\frac{1}{p}}\|f\mid L^{p}(\Omega')\|+|\Omega'|^{-\frac{1}{s}}\|f-c\mid L_{s}(\Omega')\|\biggr)|\Omega|^{\frac{1}{q}}+|\Omega|^{\frac{r-q}{r}}\|g-c\mid L^{r}(\Omega)\|.
\end{multline*}

From previous inequalities we obtain for $\varphi^{\ast}(f)=g$ finally
\begin{multline*}
\|g\mid L^{q}(\Omega)\|\leq|\Omega|^{\frac{1}{q}}|\Omega'|^{-\frac{1}{p}}\|f\mid L^{p}(\Omega')\|\\
+A_{r,s}(\Omega)K_{p,q}(\Omega)M|\Omega|^{\frac{1}{q}}|\Omega'|^{-\frac{1}{p}}\|f\mid L^{1,p}(\Omega')\|\\
+K_{p,q}(\Omega)M|\Omega|^{\frac{r-q}{r}}\|f\mid L^{1,p}(\Omega)\|.
\end{multline*}
 Therefore the composition operator 
\[
\varphi^{\ast}:W^{1,p}(\Omega')\to W^{1,q}(\Omega)
\]
 is bounded. 
\end{proof}

\section{Poincar\'e-Sobolev inequalities }

\subsection{Weighted Poincare-Sobolev inequalities}

Let $\Omega\subset\mathbb{R}^{2}$ be a planar domain and let $v:\Omega\to \mathbb R$
be a smooth positive real valued function in $\Omega$. 
For $1\leq p<\infty$  consider the
weighted Lebesgue space 
\[
L^{p}(\Omega,v):=\left\{ f:\Omega\to R:\| f\mid {L^{p}(\Omega,v)}\|:=\left(\iint_{\Omega}|f(x,y)|^{p}v(x,y)~dxdy\right)^{1/p}<\infty \right\}.
\]
It is a Banach space for the norm $\|f\mid{L^{p}(\Omega,v)}\|$.

Applications of the conformal mappings theory  to the Poincar\'e-Sobolev inequalities in planar domains is based on the following result (Theorem 3.3, Proposition 3.4 \cite{GU3}) which connected the classical mappings theory and the Sobolev spaces theory.

\begin{thm}
\label{thm:InverseCompL} Let $\Omega\subset\mathbb{R}^2$ be a simply connected domain with non-empty boundary and $\varphi : \Omega\to\mathbb D$ be a conformal homeomorphism. Suppose that the (Inverse) Brennan's Conjecture holds for the interval $[\alpha_0,2/3)$ where $\alpha_{0}\in\big(-2,0\big)$ and $p\in\big(\frac{|\alpha_0|+2}{|\alpha_0|+1},2\big)$.

Then the inverse mapping $\varphi^{-1}$ induces a bounded composition operator 
\[
(\varphi^{-1})^{\ast}:{L_{p}^{1}}(\Omega)\to{L_{q}^{1}}(\mathbb{D})
\]
for any $q$ such that 
\[
1\leq q\leq \frac{p\left|\alpha_{0}\right|}{2+\left|\alpha_{0}\right|-p}<\frac{2p}{4-p}
\]
and for any function $g\in L^1_p(\Omega)$ the inequality
$$
\|(\varphi^{-1})^{\ast} f\mid L^{1,{q}}(\mathbb D)\|\leq \left(\iint\limits_{\mathbb D} |(\varphi^{-1})'|^{\frac{(p-2)q}{p-q}}~dudv\right)^{\frac{p-q}{pq}}\|f\mid L^{1,{p}}(\Omega)\|.
$$
\end{thm}

\begin{rem}
Let us remark that $\frac{|\alpha_0|+2}{|\alpha_0|+1}>\frac{4}{3}$ for any $\alpha_{0}\in\big(-2,0\big)$.
\end{rem}

Using this theorem we prove

\vskip 0.5cm
%\begin{thm}
%\label{thm:PoincareEnWeCompP} 
{\bf Theorem C.} {\it Suppose that $\Omega\subset\mathbb{C}$ is a simply connected domain with non empty boundary, the Brennan's Conjecture holds for the interval $[\alpha_0,2/3)$, where $\alpha_{0}\in\big(-2,0\big)$ and $h(z)=J_{\varphi}(z)$ is the conformal weight defined by a conformal homeomorphism $\varphi : \Omega\to\mathbb D$. Then for every $p \in \left( {(|\alpha_0|+2)}/{(|\alpha_0|+1)},2 \right)$ and every function $f\in W^{1,p}(\Omega)$, the inequality 
$$
\inf_{c\in\mathbb R}\biggl(\iint\limits _{\Omega}|f(x,y)-c|^{r}h(z)\, dxdy \biggr)^{\frac{1}{r}}\leq B_{r,p}(\Omega,h)\biggl(\iint\limits _{\Omega}
|\nabla f(x,y)|^{p}~dxdy\biggr)^{\frac{1}{p}}
$$
holds for any $r$ such that 
\[
1\leq r\leq \frac{2p}{2-p}\cdot\frac{\left|\alpha_{0}\right|}{2+\left|\alpha_{0}\right|}<\frac{p}{2-p}
\]
with the constant 
$$
B_{r,p}(\Omega,h)\leq \inf\limits_{q\in[1,2p/(4-p))}\left\{K_{p,q}(\mathbb D)\cdot B_{r,q}(\mathbb D)\right\}.
$$}
%\end{thm}
\vskip 0.5cm

Here $B_{r,q}(\mathbb D)$ is the best constant in the (non-weighted) Poincar\'e-Sobolev inequality in the unit disc 
$\mathbb D\subset\mathbb C$ and $K_{p,q}(\Omega)$ is the norm of composition operator 
$$
\left(\varphi^{-1}\right)^{\ast}: L^{1,p}(\Omega)\to L^{1,q}(\mathbb D)
$$
generated by the inverse conformal mapping $\varphi^{-1}: \mathbb D\to\Omega$:
$$
K_{p,q}(\Omega)\leq \left(\iint\limits_{\mathbb D}|\left(\varphi^{-1}\right)'|^{\frac{(p-2)q}{p-q}}~dudv\right)^{\frac{p-q}{pq}}.
$$

\begin{proof}
By the Riemann Mapping Theorem, there exists a conformal mapping
$\varphi:\Omega\to\mathbb{D}$, and by the (Inverse) Brennan's Conjecture, 
\[
\int\limits _{\mathbb{D}}|(\varphi^{-1})^{\prime}(u,v)|^{\alpha}~dudv<+\infty,\quad\text{for all}\quad-2<\alpha_0<\alpha<2/3.
\]
Hence, by Theorem \ref{thm:InverseCompL}, 
the inequality
\[
\|\nabla(f\circ\varphi^{-1})\mid L^{q}(\mathbb{D})\|\leq K_{p.q}(\mathbb D)\| \nabla f\mid L^{p}(\Omega)\|
\]
holds for every 
function $f\in L^{1,p}(\Omega)$ and for any $q$ such that 
\begin{equation} \label{elem}
1\leq q\leq \frac{p\left|\alpha_{0}\right|}{2+\left|\alpha_{0}\right|-p}<\frac{2p}{4-p}.
\end{equation}
Choose arbitrarily $f\in C^{1}(\Omega)$. Then
  $g=f\circ\varphi^{-1}\in C^{1}(\mathbb{D})$ and, by the classical Poincar\'e-Sobolev inequality,
\begin{equation}
\inf\limits_{c\in\mathbb R}\|f\circ\varphi^{-1}-c\mid L^{r}(\mathbb{D})\|\leq B_{q,r}(\mathbb D)\|\nabla (f\circ\varphi^{-1})\mid L^{q}(\mathbb{D})\| \label{eq:PS}
\end{equation}
for any $r$ such that 
\[
1\leq r\leq \frac{2q}{2-q}
\]

By elementary calculations from the inequality (\ref{elem}) follows

\[
\frac{2q}{2-q}\leq \frac{2p}{2-p}\cdot\frac{\left|\alpha_{0}\right|}{2+\left|\alpha_{0}\right|}<\frac{p}{2-p}
\]

Combining inequalities for $2q/(2-q)$ and $r$ we conclude that the inequality (\ref{eq:PS}) holds for any $r$ such that

\[
1\leq r\leq \frac{2p}{2-p}\cdot\frac{\left|\alpha_{0}\right|}{2+\left|\alpha_{0}\right|}<\frac{p}{2-p}.
\]

Using the change of variable formula, the classical Poincar\'e-Sobolev inequality for the unit disc
$$
\inf\limits_{c\in \mathbb R}\biggl(\iint\limits_{\mathbb D}|g(u,v)-c|^r ~dudv\biggr)^{\frac{1}{r}}
\leq B_{r,q}(\mathbb D)\left(\iint\limits_{\mathbb D}|\nabla g(u,v)|^q~dudv\right)^{\frac{1}{q}}
$$
and Theorem \ref{thm:InverseCompL}, we finally infer

\begin{multline}
\inf\limits_{c\in \mathbb R}\biggl(\iint\limits_{\Omega}|f(x,y)-c|^r h(x,y)~dxdy\biggr)^{\frac{1}{r}}=\\
\inf\limits_{c\in \mathbb R}\biggl(\iint\limits_{\Omega}|f(x,y)-c|^r J_{\varphi}(x,y)~dxdy\biggr)^{\frac{1}{r}}=
\inf\limits_{c\in \mathbb R}\biggl(\iint\limits_{\mathbb D}|g(u,v)-c|^r ~dudv\biggr)^{\frac{1}{r}}\\
\leq B_{r,q}(\mathbb D)\left(\iint\limits_{\mathbb D}|\nabla g(u,v)|^q~dudv\right)^{\frac{1}{q}}
\leq K_{p,q}(\mathbb D)\cdot  B_{r,q}(\mathbb D)
\left(\iint\limits_{\Omega}|\nabla f(x,y)|^p~dxdy\right)^{\frac{1}{p}}.
\nonumber
\end{multline}

Approximating an arbitrary function $f\in W^{1,p}(\Omega)$ by smooth functions we obtain
$$
\inf_{c\in\mathbb R}\biggl(\iint\limits _{\Omega}|f(x,y)-c|^{r}h(z)\, dxdy \biggr)^{\frac{1}{r}}\leq B_{r,p}(\Omega,h)\biggl(\iint\limits _{\Omega}
|\nabla f(x,y)|^{p}~dxdy\biggr)^{\frac{1}{p}}
$$
with the constant 
$$
B_{r,p}(\Omega,h)\leq \inf\limits_{q: q\in[1,2p/(4-p))}\left\{K_{p,q}(\mathbb D)\cdot B_{r,q}(\mathbb D)\right\}.
$$
\end{proof}

The property of the conformal $\alpha$-regularity implies the integrability of a Jacobian of conformal mappings (conformal weights) and therefore for any conformal $\alpha$-regular domain we have the embedding of weighted Lebesgue spaces $L^r(\Omega,h)$ into 
non-weighted Lebesgue spaces $L^s(\Omega)$ for $s=\frac{\alpha-2}{\alpha}r$.

\begin{lem}
\label{lem:rsemb} 
Let $\Omega$ be a conformal $\alpha$-regular domain. Then for any function $f\in L^r(\Omega,h)$, $\alpha/(\alpha-2)\leq r<\infty$,  the inequality
$$
\|f\mid L^s(\Omega)\|\leq \left(\iint\limits_{\mathbb D}\left|(\varphi^{-1})^{\prime}\right|^{\alpha}~dudv\right)^{\frac{2}{\alpha}\cdot\frac{1}{s}}
\|f\mid L^r(\Omega,h)\|
$$
holds for $s=\frac{\alpha-2}{\alpha}r$.
\end{lem}

\begin{proof}
Because $\Omega$ is a conformal $\alpha$-regular domain then there exists a conformal mapping $\varphi: \Omega\to\mathbb D$ such that
$$
\left(\iint\limits_{\mathbb D}\left|J_{\varphi^{-1}}(u,v)\right|^{\frac{r}{r-s}}~dudv\right)^{\frac{r-s}{rs}}=\left(\iint\limits_{\mathbb D}\left|(\varphi^{-1})^{\prime}(u,v)\right|^{\alpha}~dudv\right)^{\frac{2}{\alpha}\cdot\frac{1}{s}}<\infty,
$$
for $s=\frac{\alpha-2}{\alpha}r$.
Then
\begin{multline}
\|f\mid L^s(\Omega)\|\\=\left(\iint\limits_{\Omega}|f(x,y)|^s~dxdy\right)^{\frac{1}{s}}=
\left(\iint\limits_{\Omega}|f(x,y)|^sJ^{\frac{s}{r}}_{\varphi}(x,y)J^{-\frac{s}{r}}_{\varphi}(x,y)~dxdy\right)^{\frac{1}{s}}
\\
\leq \left(\iint\limits_{\Omega}|f(x,y)|^rJ_{\varphi}(x,y)~dxdy\right)^{\frac{1}{r}}
\left(\iint\limits_{\Omega}J^{-\frac{s}{r-s}}_{\varphi}(x,y)~dxdy\right)^{\frac{r-s}{rs}}\\=
\left(\iint\limits_{\Omega}|f(x,y)|^rh(x,y)~dxdy\right)^{\frac{1}{r}}
\left(\iint\limits_{\mathbb D}J^{\frac{r}{r-s}}_{\varphi^{-1}}(u,v)~dudv\right)^{\frac{r-s}{rs}}\\
=\left(\iint\limits_{\Omega}|f(x,y)|^rh(x,y)~dxdy\right)^{\frac{1}{r}}
\left(\iint\limits_{\mathbb D}\left|(\varphi^{-1})^{\prime}(u,v)\right|^{\alpha}~dudv\right)^{\frac{2}{\alpha}\cdot\frac{1}{s}}.
\end{multline}
\end{proof}

From Theorem C  and Lemma \ref{lem:rsemb}  follows Theorem B:

\vspace{0.5cm}

{\bf Theorem B.}
{\it Suppose that $\Omega\subset\mathbb{C}$ is a simply connected domain with non empty boundary, the Brennan's Conjecture holds for the interval $[\alpha_0,2/3)$, where $\alpha_{0}\in\big(-2,0\big)$.

Then for every $$p \in \left( \max \left\{\frac{4}{3}, \frac{2\alpha(|\alpha_0|+2)}{(2\alpha+3\alpha|\alpha_0|-4|\alpha_0|},2 \right\}\right),$$ every $s \in [1,\frac{\alpha-2}{\alpha}\frac{p}{2-p}\frac{|\alpha_0|}{2+|\alpha_0}]$ and every function 
$f\in W^{1,p}(\Omega)$, the inequality 
\begin{equation}
\inf_{c\in\mathbb R}\biggl(\iint\limits _{\Omega}|f(x,y)-c|^{s}\, dxdy \biggr)^{\frac{1}{s}}\leq B_{s,p}(\Omega)\biggl(\iint\limits _{\Omega}
|\nabla f(x,y)|^{p}~dxdy\biggr)^{\frac{1}{p}}\label{eq:NonWPI1}
\end{equation}
holds with the constant 
\begin{multline}
B_{s,p}(\Omega)\leq \|{(\varphi^{-1})^{\prime}}| L^{\alpha}(\mathbb D)\|^{\frac{2}{s}}B_{r,p}(\Omega,h)\\
 \leq \inf\limits_{q\in[1,2p/(4-p))}\left\{B_{\frac{\alpha s}{\alpha-2},q}(\mathbb D)\cdot\|{(\varphi^{-1})^{\prime}}| L^{\alpha}(\mathbb D)\|^{\frac{2}{s}}K_{p,q}(\mathbb D) \right\}.
\nonumber
\end{multline}
}
\vskip 0.5cm

\begin{proof}
The inequality (\ref{eq:NonWPI1}) immediately follows from the main inequality of Theorem C and the main inequality of Lemma \ref{lem:rsemb}. The last part of this inequality used known estimates for the constant  of the Poincar\'e-Sobolev inequality in the unit disc.

It is necessary to clarify restrictions for parameters $p,r,s$, because these restrictions do not follow directly from Theorem A and Lemma \ref{lem:rsemb}. 

By Lemma \ref{lem:rsemb} $s=(\alpha-2)/\alpha r$.
By Theorem C 
\[
1\leq r\leq \frac{2p}{2-p}\cdot\frac{\left|\alpha_{0}\right|}{2+\left|\alpha_{0}\right|}<\frac{p}{2-p}
\]

Hence

\[
1\leq s\leq \frac{\alpha-2}{\alpha}\cdot \frac{2p}{2-p}\cdot\frac{\left|\alpha_{0}\right|}{2+\left|\alpha_{0}\right|}<\frac{\alpha-2}{\alpha}\cdot \frac{p}{2-p}
\]

Because $1\leq s$ we have from this inequality that

\[
\frac{\alpha}{\alpha-2}\leq \frac{2p}{2-p}\cdot\frac{\left|\alpha_{0}\right|}{2+\left|\alpha_{0}\right|}<\frac{p}{2-p}
\]

By elementary calculations 

\[
p\geq \frac{2\alpha(2+|\alpha_0|)}{2\alpha+3\alpha|\alpha_0|-4|\alpha_0|}> \frac{\alpha}{\alpha-1}
\]

The last inequality is correct by factor that Brennan's conjecture is correct for all $\alpha: -2<\alpha<2/3$.
\end{proof}

Theorem A follows from Theorem B, using $s=p$, that is necessary for coincidence of the first nontrivial Neumann-Laplace eigenvalue and the constant in the Poincar\'e-Sobolev inequality of Theorem B.

\vspace{0.5cm}
{\bf Theorem A.} {\it 
Let $\varphi:\Omega\to \mathbb{D}$ be a conformal homeomorphism from a conformal $\alpha$-regular domain $\Omega$ to the unit disc $\mathbb{D}$ and the Brennan's Conjecture holds for the interval $[\alpha_0,2/3)$, where $\alpha_{0}\in\big(-2,0\big)$.

Then for every $$p \in \left( \max\left\{4/3,\frac{|\alpha_0|}{2+|\alpha_0|}\cdot \frac{\alpha+2)}{\alpha}\right\},2 \right)$$ the following inequality is correct
\begin{multline}
\frac{1}{\mu_p(\Omega)}\leq\\
 \inf\limits_{q\in[1,2p/(4-p))} \left\{B^{p}_{\frac{\alpha p}{\alpha-2},q}(\mathbb D)\cdot\|{(\varphi^{-1})^{\prime}}| L^{\alpha}(\mathbb D)\|^{2}\left(\iint\limits_{\mathbb D}\left|\left(\varphi^{-1}\right)^{\prime}\right|^{\frac{(p-2)q}{p-q}}~dudv\right)^{\frac{p-q}{q}} \right\}. 
\nonumber
\end{multline}}  
\vskip 0.5cm

\begin{proof}
By Lemma \ref{lem:rsemb} $r=\frac{\alpha-2}{\alpha}p$. By Theorem C 
\[
1\leq r\leq \frac{2p}{2-p}\cdot\frac{\left|\alpha_{0}\right|}{2+\left|\alpha_{0}\right|}<\frac{p}{2-p}
\]
Hence
\[
\frac{\alpha-2}{\alpha}\leq \frac{1}{2-p}\cdot\frac{2\left|\alpha_{0}\right|}{2+\left|\alpha_{0}\right|}
\]
By elementary calculations
\[
p\geq 2-\frac{2\left|\alpha_{0}\right|}{2-\left|\alpha_{0}\right|}\frac{\alpha-2}{\alpha}>
2-\frac{\alpha-2}{\alpha}=\frac{\alpha+2}{\alpha}
\]

The last inequality is correct by factor that Brennan's conjecture is correct for all $\alpha: -2<\alpha<2/3$.
\end{proof}

\vspace{0.5cm}

{\bf Corollary A.} {\it Suppose that $\Omega$ is smoothly bounded Jordan domain with a boundary $\partial\Omega$ of a class $C^1$ with a Dini continuous normal. Let $\varphi:\Omega\to\mathbb{D}$ be a conformal homeomorphism from $\Omega$ onto the unit disc $\mathbb{D}$. Then for every 
$p \in \left( 1,2 \right)$  the following inequality is correct
$$
\frac{1}{\mu_p(\Omega)}\leq
\|{(\varphi^{-1})^{\prime}}| L^{\infty}(\mathbb D)\|^{p}\frac{1}{\mu_p(\mathbb D)}.
$$
}  

\begin{proof}
If $\Omega$ is smoothly bounded Jordan domain with a boundary $\partial\Omega$ of a class $C^1$ with a Dini continuous normal, then  
for a conformal mapping $\varphi:\Omega\to\mathbb D$, the derivative $\varphi'$ is bounded away from $0$ and $\infty$ \cite{P}. Hence. we can apply Theorem A in the limit case $\alpha=\infty$ and $p=q$.
Then
\begin{multline}
\frac{1}{\mu_p(\Omega)}\leq B^{p}_{p,p}(\mathbb D)\cdot\|{(\varphi^{-1})^{\prime}}| L^{\infty}(\mathbb D)\|^{2}
\|{(\varphi^{-1})^{\prime}}| L^{\infty}(\mathbb D)\|^{p-2}\\
=|{(\varphi^{-1})^{\prime}}| L^{\infty}(\mathbb D)\|^{p}\frac{1}{\mu_p(\mathbb D)}. 
\nonumber
\end{multline}
\end{proof}

As an application we obtain the lower estimate of the first non-trivial eigenvalue on the Neumann eigenvalue problem for the $p$-Laplace operator in the interior of the cardioid (which is a non-convex domain with a non-smooth boundary). 

\begin{exa}
\label{exa:cardioid}
Let $\Omega_c$ be the interior of the cardioid $\rho=2(1+cos\theta)$. The diffeomorphism
$$
z=\psi(w)=(w+1)^2,w=u+iv,
$$
is conformal and maps the unit disc $\mathbb D$ onto $\Omega_c$. 
Then by Theorem A:
\begin{multline}
\frac{1}{\mu_p(\Omega_c)}\leq 
\inf\limits _{1\leq q\leq\frac{2p}{4-p}}
\left(\frac{2}{\pi^{\delta}}\left(\frac{1-\delta}{{1}/{2}-\delta}\right)^{1-\delta}\right)^p
\\
\times
\left(\iint\limits_{\mathbb D}(2|w+1|)^{\alpha}~dudv\right)^{\frac{2}{\alpha}}
\left(\iint\limits_{\mathbb D}(2|w+1|)^{\frac{(p-2)q}{p-q}}~dudv\right)^{\frac{p-q}{q}}. 
\label{eq:card}
\end{multline}
Here $\delta=1/q-(\alpha-2)/\alpha p$.
\end{exa}

\noindent Vladimir Gol'dshtein \, \hskip 3.2cm Alexander Ukhlov

\noindent Department of Mathematics \hskip 2.25cm Department of Mathematics

\noindent Ben-Gurion University of the Negev \hskip 1.05cm Ben-Gurion
University of the Negev

\noindent P.O.Box 653, Beer Sheva, 84105, Israel \hskip 0.7cm P.O.Box
653, Beer Sheva, 84105, Israel

\noindent E-mail: vladimir@bgu.ac.il \hskip 2.5cm E-mail: ukhlov@math.bgu.ac.il
\end{document}